\documentclass[12pt]{article}
\usepackage{amssymb, amsthm}
\usepackage{pdfsync}
\usepackage[tbtags]{amsmath}
\usepackage{hyperref}
\usepackage{pdfsync}

\newcommand {\lt}{\left}

\newcommand {\rt}{\right}

\newcommand {\zd}{\bf{Z^d}}

\newcommand {\la}{\lambda}

\begin{sloppypar}
   
\end{sloppypar}

\newtheorem{cor}{Corollary}[section]

\newtheorem{thm}{Theorem}[section]
\newtheorem{lemma}{Lemma}[section]

\begin{document}

\title{\bf{ On the Critical Behavior of a Homopolymer Model}}

\author{ M.  Cranston , \,    S.  Molchanov} 
\maketitle
\footnote{Research of the authors supported in part by a 
grant from NSF } 
\begin{abstract}
Taking $P^0$ to be the measure  induced by simple, symmetric nearest neighbor continuous time random walk on ${\bf{Z^d}}$ starting at $0$ with jump rate $2d$ define, 
for $\beta\ge 0,\,t>0,$  the Gibbs probability measure $P_{\beta,t}$ by specifying its density with respect to $P^0$ as
\begin{eqnarray}
\frac{dP_{\beta,t}}{dP^0}=Z_{\beta,t}(0)^{-1}e^{\beta \int_0^t\delta_0(x_s)ds}
\end{eqnarray}
where $Z_{\beta,t}(0)\equiv E^0[e^{\beta \int_0^t\delta_0(x_s)ds}].$ This Gibbs probability measure provides a simple model for a homopolymer with an attractive potential at the origin. In a previous paper \cite{CM07}, we showed that for dimension $d\ge3$ there is a phase transition in the behavior of these paths from diffusive behavior for $\beta$  below a critical parameter to positive recurrent behavior for $\beta$  above this critical value. This corresponds to a transition from a diffusive or stretched out phase to a globular  phase for the polymer.  The critical value was determined by means of the spectral properties of the operator $\Delta+\beta\delta_0$ where $\Delta$ is the discrete Laplacian on ${\bf{Z^d}}.$ In this paper we give a description of the polymer at the critical value where the phase transition takes place. The behavior at the critical parameter is in some sense midway between the two phases and dimension dependent. 
\end{abstract}



\section{Introduction}

In this paper we provide a picture of a homopolymer model at the critical value of a parameter. Other properties of this model were discussed in the previous work \cite{CM07} of the authors. An approach to some of the results of  \cite{CM07} using renewal theory is explained in \cite{G}. Using the methods of \cite{CM07}, we can now give a fairly complete description of the polymer behavior at the critical parameter in all dimensions. Interest in polymer models is wide spread. An early work on the subject is \cite{MF} which has been followed by myriad contributions and we refer the reader to the interesting paper {\cite{LGK} or the monograph \cite{G} and their extensive bibliographies. In contrast to other work on the homogeneous pinning model, our approach uses spectral theory and resolvent analysis in place of renewal theory. In our case we are able to establish some interesting and explicit results about the behavior of the pinned homopolymer at the critical parameter which gives the point where a phase transition occurs. In order to describe the model,  denote by $\Sigma$ the space right continuous, left limit  paths  on  $[0,\infty)$ into ${\bf{Z^d}}.$ A typical element of $\Sigma$ will have its position at time $s$ denoted by $x_s.$ Sometimes we shall consider the restriction of elements of $\Sigma$ to the interval $[0,t]$ and use $\Sigma_t$ to denote this set of paths. The corresponding Borel $\sigma-$fields shall be denoted by $\mathcal{B}_\infty$ and $\mathcal{B}_t,$ respectively. Our reference measure on $\Sigma$ shall be $P^0$, where $P^x$ denotes the measure induced by simple, symmetric nearest neighbor continuous time random walk on ${\bf{Z^d}}$ starting at $x$ with jump rate $2d.$ This is the Markov process whose generator is the discrete Laplacian; $\Delta \psi(x)=\sum_{y:|x-y|=1}(\psi(y)-\psi(x)).$ 
Define for $\beta\ge 0$ and $t>0$ the Gibbs probability measure $P_{\beta,t}$ on the Borel subsets  $\mathcal{B}_t$ of $\Sigma_t$ by specifying its density with respect to $P^0$ as
\begin{eqnarray*}
\frac{dP_{\beta,t}}{dP^0}=Z_{\beta,t}(0)^{-1}e^{\beta \int_0^t\delta_0(x_s)ds}
\end{eqnarray*}
where 
$$Z_{\beta,t}(0)\equiv E^0[e^{\beta \int_0^t\delta_0(x_s)ds}]$$ 
is the usual normalizing factor called the partition function.
Setting 
$$Z_{\beta,t}(x)=E^x[ e^{\beta\int_0^t\delta_0(x_s)ds}]$$ 
enables us to define the Gibbs measure $P_{\beta,t}^x$ on paths started at $x$ by
\begin{eqnarray*}
\frac{dP_{\beta,t}^x}{dP^x}=Z_{\beta,t}(x)^{-1}e^{\beta \int_0^t\delta_0(x_s)ds}
\end{eqnarray*}
When $x=0$ we will write $Z_{\beta,t}$ in place of $Z_{\beta,t}(0).$ A couple of simple observations to be used later are that
\begin{eqnarray}\label{zerobest}
Z_{\beta,t}(x)\le Z_{\beta,t}(0), \,x\in{\bf{Z^d}}
\end{eqnarray}
and that for all $x\in{\bf{Z^d}}$
\begin{eqnarray}\label{tincr}
Z_{\beta,s}(x)\le Z_{\beta,t}(x),\,\mbox{for}\,s\le t.
\end{eqnarray}
In the previous work \cite{CM07}, we demonstrated the existence of a phase transition in the parameter $\beta$ at a particular parameter value which we shall denote by $\beta_d.$  It was also shown in \cite{CM07} that in all dimensions, whenever $\beta>\beta_d,$ there is a limiting measure $P_{\beta,\infty}^x.$ To be precise,  for any $A\in \mathcal{B}_s,$ the limit $\lim_{t\to\infty}P_{\beta,t}^x(A)=P_{\beta,\infty}^x(A)$ determines a probability measure $P_{\beta,\infty}^x$ on  $\mathcal{B}_\infty.$ Similar results on the existence of such a measure were obtained in the context of penalization in \cite{RVY}. It was also shown that $\{x_s:s\ge 0\}$  is positive recurrent, called the globular phase, under the measure $P_{\beta,\infty}^x.$ For $d=1,\,2$ the process under the polymer measure $P_{\beta,\infty}^x$ is null recurrent for $\beta=0$ and positive recurrent for $\beta>0.$ In other words, $\beta_d=0$ for $d=1$ or $2.$ However, in dimensions $d\ge3,$ it was shown in \cite{CM07}  for $\beta<\beta_d$ that $P_{\beta,t}(\frac{x_t}{\sqrt{t}}\in \cdot)$ has a normal  limit.  In other words, for $\beta<\beta_d$ there is no long term effect of the potential on the polymer as this is the same limit satisfied by the process under $P^0.$

In this paper we shall extend these results to the case $\beta=\beta_d$ and $d\ge3$ which is more delicate than the case $\beta\not=\beta_d.$  For $\beta=\beta_d$ and $d\ge5$ the measure $P_{\beta_d,\infty}^x$ exists and the process is positive recurrent under the measure $P_{\beta_d,\infty}^x.$ At $\beta=\beta_d$ in dimensions $3$ and $4,$ the measure $P_{\beta_d,t}(\frac{x_t}{\sqrt{t}}\in \cdot)$  has a limit which is a mixture of Gaussians. Thus, in contrast to the case $\beta<\beta_d,$  there is a long term influence of the potential on the polymer at the critical value $\beta=\beta_d.$  In dimensions $d=3$ or $4$, the polymer paths at $\beta=\beta_d$ exhibit unusual behavior midway between the cases of $d\le2$ and $d\ge5$ and making these dimension dependent behaviors explicit is the subject of this paper.

The phase transition in the polymer model corresponds to a transition for the operator $H_\beta=\Delta  +\beta \delta _0.$  In dimensions $d=1$ or $2,$ this operator has a positive eigenvalue $\lambda_0(\beta)>0$ for all $\beta>0.$ In dimensions $d\ge 3,$ for $\beta>\beta_d,\,H_\beta$ has a positive eigenvalue $\lambda_0(\beta)$ but only absolutely continuous spectrum for $\beta<\beta_d.$  We denote the corresponding eigenfunction by $\psi_{\beta}.$ Curiously, for $d\ge5,\,\lambda_0(\beta_d)=0$ is an eigenvalue at the edge of the absolutely continuous part of the spectrum of $H_{\beta_d}$ which is $[-4d,0].$  However, the situation in $d\le4$ is that there is no eigenvalue in the spectrum of $H_{\beta_d}.$ 
We remark that the value $\beta_d$ marks a transition in the free energy as well. Namely, for $\beta\le\beta_d$ the free energy $\lim_{t\to\infty}\frac{1}{t}\ln Z_{\beta,t}(0)=0$ while for $\beta>\beta_d$ one has $\lim_{t\to\infty}\frac{1}{t}\ln Z_{\beta,t}(0)=\lambda_0(\beta)>0.$ In addition, this corresponds to the fact that $\int_0^\infty \delta_0(x_s)ds$ is an exponentially distributed random variable with parameter $2d\,e_d$ where $e_d=P^0( x\,\, \mbox{never\, returns\, to\, the\, origin}).$ One simply notes that $\int_0^\infty \delta_0(x_s)ds=\sum_{j=1}^N \tau_j$ where $N,$ the number of visits to the origin, is a geometric random variable with parameter $1-e_d$ which is independent of the $iid$ sequence $\tau_j,\,j\ge1,$ of exponentially distributed random variables with parameter $2d$ where $\tau_j$ is time spent at the origin on the $j^{th}$ visit there. Thus $Z_{\beta,\infty}(0)=E^0[e^{\beta \int_0^\infty \delta_0(x_s)ds}]<\infty$ for $\beta < 2d\,e_d$ while this is infinite for $\beta \ge2d\,e_d$ which gives that $\beta_d=2d\,e_d.$

\section{Behavior of the polymer at $\beta=\beta_d.$}

We now describe the behavior of the polymer at $\beta=\beta_d.$ For $d=1,2,\, \beta_d=0$ and so the polymer measure $P_{\beta_d,\infty}$ is just $P^0$ and there is nothing new to add, but perhaps it's worth pointing out that the continuous time, simple, symmetric random walk is null recurrent in these dimensions. That is the polymer in this case is diffusive meaning $x_t/\sqrt{t}$ has a non-degenerate limiting distribution, which is of course Gaussian. 

For $d=3$ or $4,$ the polymer is in a "weakly" diffusive phase at $\beta=\beta_d.$ The potential has a weak, yet non-negligible, long term  effect in these dimensions at the critical value of the parameter $\beta,$ but not strong enough to give a stationary probability distribution under $P_{\beta_d,\infty}$ as in the case $d\ge 5$ described below. The effect of the potential shows up in the behavior of $\sigma_t/t$ where 
$$\sigma_t=\sup\{s\le t: x_s=0\}.$$ 

The process $\sigma_t/t$ has a limiting distribution under $P_{\beta_d,t}$ as $t\to\infty.$ This distribution is more concentrated near $0$ in three dimension than in four. For example, the mean of this limiting distribution is $1/3$ when $d=3$ and is $1/2$ when $d=4.$
We can derive the limiting distribution explicitly as well as that of $x_t/\sqrt{t}$ with respect to $P_{\beta_d,t}$ as $t\rightarrow \infty.$ The latter relies on specifying the limiting distribution of $\sigma_t/t.$ In the limit, the distribution of $x_t/\sqrt{t}$ in the critical case $\beta=\beta_d$ is a mixture of Gaussians. The reason is that the polymer is "free" of the influence of the potential after time $\sigma_t$ and as a result, conditional on $\sigma_t,$ the position $x_t$ is approximately Gaussian with variance $t-\sigma_t=t(1-\sigma_t/t).$ One can think of the potential as providing a "sticky" boundary point in the critical case, but not "sticky" enough to create a globular phase as in the cases $d\ge5.$ The existence of a non-degenerate limiting distribution for $\sigma_t/t$ under the measure $P_{\beta_d,t}$ demonstrates the long term influence of the potential on the polymer.

In dimensions $d\ge5,$  the process under the polymer measure $P_{\beta_d,\infty}^x$ is positive recurrent. This curious case is due to the existence of $0$ as an eigenvalue for the operator $H_{\beta_d}.$ It turns out that the corresponding eigenfunction, $\psi_{\beta_d}$ is given by  $\psi_{\beta_d}(x)=P^x(T<\infty)$ where $T=\inf\{t\ge 0:x_t=0\}.$  For dimensions $d\ge5,$  the normalized square of $\psi_{\beta_d}$ provides the stationary probability measure for the time-homogeneous Markov process under $P_{\beta_d}.$ The stationary distribution is given by $\pi_{\beta_d}$ where for $\beta\ge \beta_d$
\begin{eqnarray}\label{pie}
\pi_{\beta}=\sum_{x\in{\bf{Z^d}}}\frac{\psi^2_{\beta}(x)}{||\psi_{\beta}||^2_{L^2({\bf(Z^d}})}\delta_x.
\end{eqnarray}
Note that this definition makes sense for $d=3$ or $4$ when $\beta>\beta_d$ and for $d\ge5$ when $\beta\ge\beta_d.$ 
This measure has fairly heavy tails when $\beta=\beta_d,$ in the sense that only moments of order up to $d-5$ exist.  By contrast, in the case $\beta>\beta_d,$ all moments exist for the measure $\pi_{\beta}.$ This is why we say the polymer is in the "weakly" globular phase at $\beta=\beta_d$ for $d\ge5.$
For $d=3$ and $4,$ the following Theorem describes the behavior at criticality.
\begin{thm}\label{maintheorem}
For $d=3$ and $\beta = \beta_3$ and $d_3=\frac{16 \sqrt\pi}{\beta_3},$ as $t \rightarrow \infty,$
\begin{eqnarray*}
\begin{split}
Z_{\beta_3,t}&\sim d_3 \sqrt{ t}, \\ 
E_{\beta_3,t}\left[\lt<\zeta,x_t\rt>^2\right]&\sim\frac{2}{3} |\zeta|^2t,\, \zeta\in {\bf{R^3}},\\ 
Cov_{\beta_3,t}\left(x_t\right)&\sim\frac{2}{3} t I,\\  
P_{\beta_3,t}(\sigma_t/t\in du)&\rightarrow \frac{1}{2\sqrt{u}}du,\,0\le u\le 1,\\
 P_{\beta_3,t} \left(\frac{x_t}{\sqrt t} \in \cdot\right)&\rightarrow P\left(\xi \in \cdot\right)
\end{split}
\end{eqnarray*}
where $\xi$ is a random vector with characteristic function
$$\psi_\xi(\phi)=\frac{1}{2}\int_0^1e^{-|\phi|^2(1-u)}\frac{1}{\sqrt{u}} du,\,\,\phi\in {\bf{R^3}}.$$

For $d=4,$ and $\beta=\beta_4,$ and $d_4=\frac{8\pi}{\beta_4},$ as $t \rightarrow \infty,$
\begin{equation*}\begin{split}
Z_{\beta_4,t}&\sim d_4\frac{ t}{\ln t}, \\ 
E_{\beta_4,t}\left[\lt<\zeta,x_t\rt>^2\right]&\sim |\zeta|^2t,\, \zeta\in {\bf{R^4}},\\ 
Cov_{\beta_4,t}\left(x_t\right)&\sim t I,\\  
P_{\beta_4,t}(\sigma_t/t\in du)&\rightarrow du,\,0\le u\le 1,\\
 P_{\beta_4,t} \left(\frac{x_t}{\sqrt t} \in \cdot\right)&\rightarrow P\left(\eta \in \cdot\right)
 \end{split}\end{equation*}
where $\eta$ is a random vector with characteristic function
$$\psi_\eta(\phi)=\int_0^1e^{-|\phi|^2(1-u)} du,\,\,\phi\in {\bf{R^4}}.$$
\end{thm}

For the next result, denote the heat kernel of the operator $H_\beta$ by $p_\beta.$ That is
\begin{eqnarray}\label{pbeta}
\frac{\partial p_\beta}{\partial t}(t,x,y)= H_\beta p_\beta(t,x,y),\,\,p_\beta(0,x,y)=\delta_x(y).
\end{eqnarray}

The situation at criticality for $d\ge5$  is described in the following. 
\begin{thm}\label{globular}
For $d\ge 5$ and $\beta=\beta_d,$ there is a measure $P_{\beta_d,\infty}$ on $(\Sigma,\mathcal{B})$ such that for each fixed $s,$ as $t\to \infty$ the process $\left(\{x_u:0\le u\le s \},\,P_{\beta_d,t}\right)$ converges in law to $\left(\{x_u:0\le u\le s\},\,P_{\beta_d,\infty}\right).$ The process $\left(\{x_s:0\le s<\infty\},\,P_{\beta_d,\infty}\right)$ is a Markov process with generator
$$A_{\beta_d} f(x)=\sum_{|y-x|=1} a_{d}(x,y)(f(y)-f(x)),$$
where
\begin{eqnarray*}
\begin{split}\label{hpath}
a_{d}(x,y)=\left \{  \begin{array}{lll}
0,&\mbox{if $|x-y|>1,$}\\
\frac{\psi_{\beta_d}(y)}{\psi_{\beta_d}(x)},&\mbox{if $|x-y|=1,$}\\
2d-\beta_d\delta_0(x),&\mbox{if $y=x,$}
\end{array}
\right.
\end{split}
\end{eqnarray*}
and $\psi_{\beta_d}$ denotes the eigenfunction of $H_{\beta_d}$ normalized so that $\psi_{\beta_d}(0)=1.$
  \label{invthm}
The transition probabilities for this ergodic, pure jump,  Markov process on $\bf{Z^d}$ are given by 
\begin{eqnarray}\label{transdens}
r_{\beta_d}(s,x,y)&=&\frac{p_{\beta_d}(s,x,y)\psi_{\beta_d}(y)}{\psi_{\beta_d}(x)}.
\end{eqnarray}
Its invariant probability distribution is  $\pi_{\beta_d}$ as defined in (\ref{pie}).
The $k^{th}$ moment of $\pi_{\beta_d}$ is finite if and only if $d\ge k+5.$


\end{thm}

\section{Resolvent Analysis}

Assume throughout that $\beta\ge 0.$ Recall that  $p_\beta$  denotes the heat kernel of the operator $H_\beta$ as at (\ref{pbeta}) where $H_0=\Delta.$
The analysis begins with the resolvent given by
\begin{eqnarray}\label{rdef}
R_{\beta,\la}(x,y)=\int_0^\infty e^{-\la s}p_\beta(s,x,y)ds,\label{defR0}
\end{eqnarray}
The resolvent satisfies the equation
\begin{eqnarray}\label{reseqn}
\left(H_\beta-\la\right)R_{\beta,\la}(x,y)=-\delta_y(x)
\end{eqnarray}
For $\phi \in {\bf{T^d}},$ the $d-$dimensional torus, with coordinates $\phi=(\phi_1,...\phi_d),$ use 
\begin{eqnarray}\label{phi}
\Phi(\phi)=2\sum_{j=1}^d(1-\cos \phi_j)
\end{eqnarray}
to denote the symbol (Fourier transform) of $-\Delta.$ 

Using (\ref{reseqn}), we see that the Fourier transform of the resolvent, namely,
\begin{eqnarray*}
\hat{R}_{\beta,\lambda}(\phi,y)=\sum_{y\in {\bf{Z^d}}}R_{\beta,\lambda}(x,y)e^{i<\phi,x>},
\end{eqnarray*}
satisfies the equation
\begin{eqnarray*}
-\hat{R}_{\beta,\lambda}(\phi,y)\left(\Phi(\phi)+\la\right)+\beta R_{\beta,\lambda}(0,y)=-e^{i\lt<\phi,y\rt>}.
\end{eqnarray*}
Solving for $\hat{R}_{\beta,\lambda}(\phi,y)$ one arrives at
\begin{eqnarray}\label{Rhat}
\hat{R}_{\beta,\lambda}(\phi,y)=\frac{\beta R_{\beta,\lambda}(0,y)+e^{i\lt<\phi,y\rt>}}{\lambda+\Phi(\phi)}.
\end{eqnarray}

In the case $\beta=0,$ (\ref{Rhat}) becomes

\begin{eqnarray}\label{rhat}
\hat{R}_{0,\la}(\phi,y)=\frac{e^{i<\phi,y>}}{\la+\Phi(\phi)}.
\end{eqnarray}
On inversion of (\ref{rhat}), we have the representation
\begin{eqnarray}\label{rrep}
R_{0,\la}(0,y)=\frac{1}{(2\pi)^d}\int_{{\bf{T^d}}}\frac{e^{i<\phi,y>}}{\la+\Phi(\phi)}d\phi.
\end{eqnarray}

Since the function $R_{0,\la}(0,0)$ plays a central role in our development, we denote it for simplicity by
\begin{eqnarray}\label{Ilam}
\begin{split}
I(\lambda)=&\frac{1}{(2\pi)^d}\int_{{\bf{T^d}}}\frac{1}{\la+\Phi(\phi)}d\phi
\end{split}
\end{eqnarray}
Notice that $I(0)=\infty$ for $d=1,2$ but $I(0)<\infty$ for $d\ge3.$

Multiplying both sides of (\ref{Rhat}) by $(2\pi)^{-d},$ integrating over $T^d$ and combining  with (\ref{rrep})  and (\ref{Ilam}), we get
 \begin{eqnarray*}
 R_{\beta,\lambda}(0,y)=\beta I(\la)R_{\beta,\lambda}(0,y)+R_{0,\lambda}(0,y)
 \end{eqnarray*}
and solving for $R_{\beta,\lambda}(0,y),$ 
 \begin{eqnarray}\label{Rlambda}
 R_{\beta,\lambda}(0,y)=\frac{R_{0,\lambda}(0,y)}{1-\beta I(\la)}.\label{rbeta}
\end{eqnarray}

The following result about $I(\lambda)$ enables one to derive large time asymptotics for $p_{\beta_d}(t,0,0)$ by means of a Tauberian Theorem.
\begin{lemma}\label{I}

For  $d=3$ and $4$ one has $I(0)<\infty$ and the following $\la\to 0$ asymptotics hold for $I(\la),$ 
\[I(\la)\sim\left\{\begin{array}{ll}

I(0)-\frac{\sqrt{\lambda}}{4\pi},\,d=3,\\
I(0)-\frac{\lambda}{8\pi} \ln \frac{1}{\lambda},\,d=4,\\
I(0)-c_d\lambda,\,d\ge 5,
\end{array}\right.\]
where $c_d$ is some positive constant.
\end{lemma}
We don't specify the value of $c_d$ as it's value won't play a significant role later.
\begin{proof} 
By (\ref{Ilam}), for $d=3,$ and $\delta>0$ fixed,
\begin{eqnarray*}
\begin{split}
I(0)- I(\la)=&\frac{\la}{(2\pi)^3}\int_{{\bf{T^3}}}\frac{1}{(\la+\Phi(\phi))\Phi(\phi)}d\phi\\
\sim&\frac{4\pi\la}{(2\pi)^3}\int_0^\delta\frac{r^2}{(\la+r^2)r^2}dr\\
\sim&\frac{\la}{2\pi^2}\int_0^\delta\frac{1}{\la+r^2}dr\\
\sim&\frac{\sqrt{\la}}{4\pi}.
\end{split}
\end{eqnarray*}
Solving for $I(\la)$ gives
\[I(\la)=I(0)-\frac{\sqrt{\la}}{4\pi}+...\]
Similarly, for $d=4,$ using (\ref{Ilam}), again with $\delta>0$ fixed,
\begin{eqnarray*}
\begin{split}
I(0)- I(\la)=&\frac{\la}{(2\pi)^3}\int_{{\bf{T^3}}}\frac{1}{(\la+\Phi(\phi))\Phi(\phi)}d\phi\\
\sim&\frac{2\pi^2\la}{(2\pi)^3}\int_0^\delta\frac{r^3}{(\la+r^2)r^2}dr\\
\sim&\frac{\la}{4\pi}\int_0^\delta\frac{r}{\la+r^2}dr\\
\sim&\frac{\la}{8\pi}\ln\frac{1}{\lambda}.
\end{split}
\end{eqnarray*}
Solving for $I(\la)$ gives
\[I(\la)=I(0)-\frac{\la}{8\pi}\ln\frac{1}{\la}+....\]

In the case $d\ge 5,$ and $\delta>0$ fixed as before,  and letting $c_d$ change from line to line,
\begin{eqnarray*}
\begin{split}
I(0)- I(\la)=&\frac{\la}{(2\pi)^d}\int_{{\bf{T^d}}}\frac{1}{(\la+\Phi(\phi))\Phi(\phi)}d\phi\\
\sim&c_d\lambda\int_0^\delta\frac{r^{d-1}}{(\la+r^2)r^2}dr\\
\sim&c_d\lambda \int_0^\delta\frac{r^{d-3}}{\la+r^2}dr\\
=&c_d\lambda^{\frac{d}{2}-1}\int_0^{\delta/\sqrt{\lambda}}\frac{s^{d-3}}{1+s^2}ds\\
\sim&c_d\lambda.
\end{split}
\end{eqnarray*}
\end{proof}

\begin{cor}
The following $\la\to 0$ asymptotics hold 

\begin{eqnarray}\label{resI}
R_{\beta_3,\lambda}(0,0)\sim \frac{4 \pi}{\beta_3^2\sqrt \lambda},\,\,\,d=3,
\end{eqnarray}
\begin{eqnarray}\label{resII}
R_{\beta_4,\lambda}(0,0)\sim\frac{8\pi}{\beta_4^2\lambda\ln\frac{1}{\lambda}},\,\,d=4.
\end{eqnarray}
\begin{eqnarray}\label{resIII}
R_{\beta_d,\lambda}(0,0)\sim\frac{c_d}{\lambda},\,\,d\ge 5.
\end{eqnarray}
\end{cor}
\begin{proof}
Using ${\bf{Lemma\,\ref{I}}},$ (\ref{Rlambda}) and the fact that  $\beta_dI(0)=1,$ for $d=3,$ the Laplace transform of $p_{\beta_3}(\cdot,0,0)$ satisfies
 \begin{eqnarray*}
 \begin{split}
 R_{\beta_3,\lambda}(0,0)&=\frac{I(\lambda)}{1-\beta_3I(\lambda)} \\ 
&\sim\frac{4\pi I(0)}{\beta_3\sqrt{\lambda}}\\ 
&=\frac{4 \pi}{\beta_3^2\sqrt \lambda},\,\, \lambda \rightarrow 0.
\end{split}
\end{eqnarray*}
Similarly, for $d=4,$ it satisfies
 \begin{eqnarray*}
 \begin{split}
 R_{\beta_4,\lambda}(0,0)&=\frac{I(\lambda)}{1-\beta_4I(\lambda)} \\ 
&\sim\frac{8 \pi I(0)}{\beta_4\lambda\ln\frac{1}{\lambda}}.\\ 
&=\frac{8\pi}{\beta_4^2\lambda\ln\frac{1}{\lambda}},\,\, \lambda \rightarrow 0.
\end{split}
\end{eqnarray*}
\end{proof}

The next result is derived by standard Tauberian arguments as in \cite{F}.
\begin{lemma}\label{plemma}
For $d=3,$ and $c_3=\frac{8 \sqrt\pi}{\beta_3^2},$  
\begin{eqnarray}\label{pthreeasym}
p_{\beta_3}(t,0,0)\sim c_3 t^{-1/2},\,\,t \rightarrow \infty.
\end{eqnarray}

For $d=4,$ and $c_4=\frac{8\pi}{\beta_4^2},$
\begin{eqnarray}\label{pfourasym}
p_{\beta_4}(t,0,0)\sim c_4 \frac{1 }{\ln t},\,\,t \rightarrow \infty.
\end{eqnarray}

For $d\ge5,$ and $c_d$ as in (\ref{resIII}),
\begin{eqnarray}\label{pdasym}
p_{\beta_d}(t,0,0)\sim  \frac{c_d }{ t},\,\,t \rightarrow \infty.
\end{eqnarray}
\end{lemma}
\begin{proof}
Set $\tau=\frac1t.$ By Theorem $1,$ page $443$ of \cite{F} and (\ref{resI}),  $\frac{R_{\beta_3,\tau\lambda}(0,0)}{R_{\beta_3,\tau}(0,0)}\sim \lambda^{-1/2},\tau\to 0$ is equivalent to $p_{\beta_3}(t,0,0)\sim \frac{R_{\beta_3,1/t}(0,0)}{\Gamma(\frac32)}.$ But the first asymptotic holds by (\ref{resI}) and so  (\ref{pthreeasym}) holds since $\Gamma(\frac32)= \frac{\sqrt{\pi}}{2}.$

By Theorem $2,$ page $445$ of \cite{F} and (\ref{resII}), 
$R_{\beta_4,\tau}(0,0)\sim \frac{8\pi}{\beta_4^2\tau\ln\frac{1}{\tau}}$ 
is equivalent to $p_{\beta_4}(t,0,0)\sim  \frac{8\pi}{\Gamma(2)\beta_4^2}\frac{1}{\ln t}.$ But the first asymptotic holds by (\ref{resII}) and so  (\ref{pthreeasym}) holds since $\Gamma(2)=1.$ 

Again, by Theorem $1,$ page $443$ of \cite{F} and (\ref{resI}),  with $\tau=\frac1t$ since$\frac{R_{\beta_d,\tau\lambda}(0,0)}{R_{\beta_d,\tau}(0,0)}\sim \lambda^{-1},\tau\to 0$ is equivalent to $p_{\beta_d}(t,0,0)\sim \frac{R_{\beta_3,1/t}(0,0)}{\Gamma(2)}\sim \frac{c_d}{t}.$ 
\end{proof}

Before moving on to the next section, we make some observations about  $ Z_{\beta,t}$ and  $p_\beta.$ First we point out the relation between $ Z_{\beta,t}$ and  $p_\beta.$   By the Feynman-Kac formula,
\begin{eqnarray}\label{Zb}
\frac{\partial Z_{\beta,t}}{\partial t}(x)=\Delta Z_{\beta,t}(x) +\beta \delta _0(x) Z_{\beta,t}(x),\,Z_{\beta,0}(x)\equiv 1.
\end{eqnarray}
Comparing (\ref{Zb}) and (\ref{pbeta}), we see that the relation between $Z_{\beta,t}$ and $p_\beta$ is given by
\begin{eqnarray}\label{Zrelp}
Z_{\beta,t}(x)=\sum_{y\in{\bf{Z^d}}}p_\beta(t,x,y).
\end{eqnarray}
Second we point out a couple simple relations satisfied by $p_\beta.$
On taking the Fourier transform in (\ref{pbeta}), it follows that
\[-\Phi(\phi)\hat{p}_\beta(t,0,\phi)+\beta p_\beta(t,0,0)=\frac{\partial \hat{p}_\beta}{\partial t}(t,0,\phi).\]
Since $\Phi(0)=0,$ it also follows that
\begin{eqnarray}\label{phatformula}
\beta p_\beta(t,0,0)=\frac{\partial \hat{p}_\beta}{\partial t}(t,0,0),\,\hat{p}_\beta(0,0,0)=1.
\end{eqnarray}

\section{Spectrum of $H_\beta$}

In this section we discuss the spectrum of the operator $H_\beta$ which plays a major role in the behavior of the measure $P_{\beta,t}.$
This discussion will be familiar to readers of the monograph \cite{CM}. The operator $H_0=\Delta$ has purely absolutely continuous (a.c.) spectrum equal to $[-4d,0].$ 
 
The spectrum of $H_\beta$ consists of an a.c. part,  $[-4d,0],$ and at most one eigenvalue $\la_0(\beta)$ which, for $d\ge 5,$  can be $0.$ In other words,  on the edge of the a.c. part of the spectrum there is an embedded eigenvalue.  The first part can be seen from the formula $(\ref{rbeta}) $ since the resolvent for $\Delta$ blows up for $\lambda\in[-4d,0]$ and the denominator vanishes when $\beta(\lambda)=1.$ That $0$ is an eigenvalue for $d\ge 5$ follows from the square integrability of $1/\Phi(\phi)$ over ${\bf{T^d}}$ which is readily verified.
Now if $H_\beta \psi=\la \psi$ with $\psi \in L^2({\bf{Z^d}})$ then $\psi$ is an eigenfunction and on taking Fourier transforms we see from a computation similar to that at (\ref{Rhat}) that
\begin{eqnarray}\label{fourierpsi}
\hat{\psi}(\phi)=\frac{\beta \psi(0)}{\la+\Phi(\phi)}.
\end{eqnarray}
Inverting gives
\begin{eqnarray}\label{psizero}
\begin{split}
\psi(0)=&\frac{\beta \psi(0)}{(2\pi)^d}\int_{{\bf{T^d}}}\frac{1}{\la+\Phi(\phi)}d\phi\\
=&\beta \psi(0)I(\la).
\end{split}
\end{eqnarray}
So, either $\psi(0)=0$ which implies $\hat{\psi}\equiv0$ and therefore $\psi\equiv0,$ or we can normalize so that $\psi(0)=1$ which implies by (\ref{fourierpsi}) that 
$\hat{\psi}(\phi)=\frac{\beta}{\la+\Phi(\phi)}$ and from (\ref{psizero}) that $\la=\la_0(\beta)$ is the solution of 
\begin{eqnarray*}
I(\la_0(\beta))=\frac{1}{\beta}.\label{deflabet}
\end{eqnarray*}
Since $\psi \in L^2({\bf{Z^d}})$ if and only if $\hat{\psi}\in L^2({\bf{T^d}})$ we see from (\ref{fourierpsi})  that $\la \notin (-4d,0)$, i.e. either $\la \le -4d$ or $\la \ge 0,$ including , apriori $\la=0$ and $\la=-4d.$ But for real $\la \le -4d,$ we have that $I(\la)<0$ and $I(\la)=\frac{1}{\beta}$ can not hold. That is, any solution of $I(\la)=\frac{1}{\beta}$ must be in $[0,\infty).$ We summarize the developments so far. 
\begin{thm}\label{spec-thm}

For $d\ge 1,$ the eigenvalue $\la_0(\beta)>0$ exists for $\beta>\beta_d= I(0)^{-1}.$  It is the root of the equation $I(\la)=\frac{1}{\beta}.$ The corresponding  unique  eigenfunction, $\psi_\beta$ has Fourier transform
\begin{eqnarray}
\hat{\psi_\beta}(\phi)=\frac{\beta}{\la_0(\beta)+\Phi(\phi)}\label{psihat}
\end{eqnarray}
and has the representation
\begin{eqnarray}
\psi_\beta(x)=\frac{\beta}{(2\pi)^d}\int_{T^d}\frac{e^{i\left<\phi,x\right>}}{\lambda_0(\beta)+\Phi(\phi)}d\phi.\label{psi}
\end{eqnarray}
The measure $\pi_\beta$ defined at (\ref{pie}) has finite moments of all orders.

For $d\ge5$ and $\beta=\beta_d,\,\lambda_0(\beta_d)=0$ is an eigenvalue. Its eigenfunction $\psi_{\beta_d}$ satisfies (\ref{psihat}) and (\ref{psi}). 
The function $1/\Phi(\phi)$ is in $L^2({\bf{T^d}})$ and so $\psi_{\beta_d}\in L^2({\bf{Z^d}}).$
The measure $\pi_{\beta_d}$ has moments of order $k$ only for $k\le d-5.$

\end{thm}  
\begin{proof}

To see that $1/\Phi(\phi)$ is in $L^2({\bf{Z^d}})$ just observe from (\ref{phi}) that $\Phi^2(\phi)\sim ||\phi||^4$ as $||\phi||\to 0$ and integrating in polar coordinates introduces a factor of $||\phi||^{d-1}$ and so $1/\Phi^{2}(\phi)$ becomes integrable for $d\ge 5.$ Thus 
$\psi_{\beta_d}(x)=\frac{\beta_d}{(2\pi)^d}\int_{\bf{T^d}}\frac{e^{-i\lt<\phi,x\rt>}}{\Phi(\phi)}d\phi$ is an $L^2({\bf{Z^d}})$ eigenfunction corresponding to the eigenvalue $0.$

The other claim is about the moments of $\pi_{\beta_d}.$ Assuming that $d\ge 5$  and $\beta>\beta_d,$ we see that for any $(j_1,j_2,\cdots,j_d)\in {\bf{N}^d},$
\[\lt(\Pi_{i=1}^d\frac{\partial^{j_i}}{\partial \phi_{i}^{j_i}}\rt)\hat{\psi}_\beta(\phi)=\lt(\Pi_{i=1}^d\frac{\partial^{j_i}}{\partial \phi_{i}^{j_i}}\rt)\frac{1}{\lambda_0(\beta)+\Phi(\phi)}\]
has moments of all orders, since the denominator is bounded from $0$ and the integration is over the compact space ${\bf{T^d}}.$ Since this is the Fourier transform of $(\Pi_{i=1}^dx_i^{j_i})\psi_\beta(x),$ Plancherel's identity implies that the latter is in $L^2$ which is the claim to be proved.   However, at $\beta=\beta_d,$ since $\lambda_0(\beta_d)=0,$ 
\begin{eqnarray*}
\begin{split}
\lt(\Pi_{i=1}^d\frac{\partial}{\partial \phi_{i}^{j_i}}\rt)\hat{\psi}_\beta(\phi)=&\lt(\Pi_{i=1}^d\frac{\partial}{\partial \phi_{i}^{j_i}}\rt)\frac{\beta_d}{\Phi(\phi)}
\end{split}
\end{eqnarray*}
The rest of the proof is just verification, for example, when $k=1,$ if we set $r_j=\sum_{i\not= j}\phi_i^2$ and $\psi_j=\phi_j/r_j$ so that we may express the asymptotic near $0$
\begin{eqnarray*}
\frac{\partial}{\partial{\phi_j}}\frac{1}{\Phi(\phi)}=&\frac{2\sin\phi_j}{\Phi(\phi)^2}\\
\sim&c\frac{\phi_j}{(\phi_j^2+\sum_{i\not= j}\phi_i^2)^2}\\
=&c\frac{\psi_j}{(\psi_j^2+1)^2}r_j^{-3/2}
\end{eqnarray*}
Thus, $\frac{\partial}{\partial{\phi_j}}\frac{1}{\Phi(\phi)}\in L^2({\bf{T^d}})$ if and only if $\frac{\psi_j}{(\psi_j^2+1)^2}r_j^{-3/2}\in L^2({\bf{T^d}})$ which depends only on the integral near $0.$ Thus, we check in cylindrical coordinates, 
\begin{eqnarray*}
\int_{\sum_{i\not= j}\phi_i^2<\delta^2} \int_0^\delta \frac{\phi_j^2}{(\phi_j^2+\sum_{i\not= j}\phi_i^2)^4}=&\int_0^\delta \int_0^{\delta/ r_j}\frac{\psi_j^2}{(\psi_j^2+1)^4}d\psi_j r_j^{d-4}{dr_j}
\end{eqnarray*}
and since $\int_0^{\delta/ r_j}\frac{\psi_j^2}{(\psi_j^2+1)^4}d\psi_j $ is bounded above and below by a constant multiple of $(\delta/r_j)^3$ the first integral converges if and only if $ \int_0^{\delta}r_j^{d-7}dr_j<\infty.$
The last integral is finite if and only if $d\ge 7.$ 
The full claim  that the $2k^{th}$ moment of $\pi_{\beta_{cr}}$ is finite if and only if $d\ge 2k+5$ is a routine (though tedious) calculation which we shall omit. 
\end{proof}

\section{Proof of (\ref{maintheorem}) and (\ref{globular})}
\begin{proof}

First for $d=3$ at $\beta=\beta_3$ by   (\ref{pthreeasym}), (\ref{Zrelp}) and (\ref {phatformula}), we have
\begin{equation}
\begin{split}
Z_{\beta_3,t}&=\sum_{x\in {\bf{Z^3}}}p_{\beta_3} (t,0,x)\\ 
&=\hat{p}_{\beta_3}(t,0,0)\\ 
&= 1+\beta_3 \int_0^t p_{\beta_3} (s,0,0)ds\\ 
&\sim 2c_3\beta_3 \sqrt{t}, \,t\rightarrow \infty.\label{Zthreeas}
\end{split}
\end{equation}
For the variance estimate, we first observe that on taking the Fourier transform of (\ref{pbeta}) and integrating from $0$ to $t$ one obtains
\begin{eqnarray}\label{pbhatpzero}
\hat{p}_{\beta_3}(t,0,\phi)=e^{-\Phi(\phi)t}+\beta_3\int_0^te^{-\Phi(\phi)(t-s)}p_{\beta_3}(s,0,0)ds
\end{eqnarray}
and on inverting this Fourier transform one arrives at 
\begin{eqnarray}\label{pbpzero}
p_{\beta_3}(t,0,x)=p_0(t,0,x)+\beta_3\int_0^t p_0(t-s,0,x)p_{\beta_3}(s,0,0)ds.
\end{eqnarray}
And since $\sum_{x\in{\bf{Z^d}}}p_0(t,0,x)\left<\zeta,x \right>^2=t|\zeta|^2$   it follows using (\ref{pthreeasym}) and (\ref{Zthreeas}) that
\begin{equation*}
\begin{split}
E_{\beta_3,t}\left[\left<\zeta,x_t\right>^2\right]&=Z^{-1}_{\beta_3,t}\sum_{x\in \bf{Z^3} }p_{\beta_3}(t,0,x)<\zeta,x>^2\\ 
&=Z^{-1}_{\beta_3,t}\left(|\zeta|^2t+\beta_3 |\zeta|^2\int_0^t (t-s)p_{\beta_3}(s,0,0)ds\right)\\ 
&\sim\left(2c_3 \beta_3\sqrt{t}\right)^{-1}|\zeta|^2t\left(1 +\beta_3 c_3 t^{1/2}\int_0^1(1-u)u^{-1/2}du \right),\,\,\ t\rightarrow \infty\\ 
&\sim \frac{2}{3} |\zeta|^2t.
\end{split}
\end{equation*}
A similar argument handles the covariance claim.
Note this implies the proper normalization of $x_t$ for a limiting law would be $\frac{x_t}{\sqrt t}.$ 

For the distribution of $\sigma_t/t$ we observe, for $s\in(0,1)$ and $\epsilon>0$ small enough so that $s+\epsilon\in (0,1)$  
\begin{eqnarray*}
\begin{split}
P_{\beta_3,t} \left(\sigma_t /t\in (s,(s+\epsilon))\right)= &6P_{\beta_3,t} \left(x_{st}=0,x_{(s+\epsilon)t}=v_1,x_u\not=0, (s+\epsilon)t\le u\le t\right)+o(\epsilon)\\
 \stackrel{\sim}{\epsilon\to0}&\frac{6E\lt[e^{\beta_3\int_0^t\delta_0(x_u)du}\delta_0(x_{st})\rt]P_{v_1}(x_u\not=0,\,0\le u\le (1-(s+\epsilon))t)\epsilon t}{Z_{\beta_3,t}}
 \end{split}
\end{eqnarray*}
where the term $\epsilon t$ arises from the rate of jumping from $0$ to the unit vector $v_1\in {\bf{Z^3}}$ in the time interval $(st,(s+\epsilon)t).$ 
As $t\to\infty,$ the term
 \begin{eqnarray*}
 P_{v_1}(x_u\not=0,\,0\le u\le (1-(s+\epsilon))t)\to e_3\in(0,1)
\end{eqnarray*}
 due to the transience of the $3$ dimensional random walk.  
Thus,
\begin{eqnarray*}
\begin{split}
P_{\beta_3,t} \left(\frac{\sigma_t}{t} \in ds\right)/ds\sim&Z^{-1}_{\beta_3,t}E\lt[e^{\beta_3\int_0^t\delta_0(x_u)du}\delta_0(x_{st})\rt]6e_3t\\
=&Z^{-1}_{\beta_3,t}p_{\beta_3}(st,0,0)6e_3t.
\end{split}
\end{eqnarray*}
By (\ref{Zthreeas}) and (\ref{pthreeasym}) we have for $s\in (0,1),$
\begin{eqnarray*}
\lim_{t\to\infty}Z^{-1}_{\beta_3,t}p_{\beta_3}(st,0,0)6e_3t=&\lim_{t\to\infty}\frac{c_3\sqrt{st}^{-1}6e_3t}{2 c_3\beta_3\sqrt{t}}\\
=&\frac{6r_3}{2\beta_3 \sqrt{s}}
\end{eqnarray*}
from which we can deduce that $6e_3=\beta_3$ and so 
\[\lim_{t\to\infty}P_{\beta_3,t} \left(\frac{\sigma_t}{t} \in ds\right)/ds=\frac{1}{2\sqrt{s}},\,\,s\in(0,1)\]
as desired.

In order to derive the asymptotic distribution of $x_t/\sqrt{t},$ we evaluate $\hat{p}_{\beta_3}$ at $\frac{\phi}{\sqrt t}$ with $\phi\in {\bf{R^3}}$ to get from the Central Limit Theorem for the simple symmetric continuous time random walk and (\ref{pbhatpzero}) that
\begin{eqnarray*}
\hat{p}_{\beta_3}(t,0,\frac{\phi}{\sqrt t})=e^{-|\phi| ^2}(1+o(1))+c_3\beta_3 \int_0^t (1+o(1))e^{-|\phi|^2(1-\frac st)}(1+s)^{-1/2}ds
\end{eqnarray*}
After normalization by $Z_{\beta_3,t}\sim2c_3 \beta_3 \sqrt t$ it follows that
$$E_{\beta_3,t} \left[e^{i<\phi,\,\frac{x_t }{\sqrt t}>}\right]=\frac{\hat{p}_{\beta_3}(t,0,\frac{\phi}{\sqrt t})}{Z_{\beta_3,t}}\rightarrow\frac12 \int_0^1e^{-|\phi|^2(1-u)}\frac{1}{\sqrt u}du.$$

Next, for $d=4$ at $\beta=\beta_4$  by (\ref{pfourasym}), (\ref{Zrelp}) and (\ref {phatformula}),  we have
\begin{equation}
\begin{split}
Z_{\beta_4,t}&= \sum_{x\in{\bf{Z^4}}}p_{\beta_4}(t,0,x)\\ 
&=\hat{p}_{\beta_4}(t,0,0)\\ 
&= 1+\beta_4 \int_0^t p_{\beta_4} (s,0,0)ds\\ 
&\sim c_4\beta_4 \frac{t}{\ln t}, \,t\rightarrow \infty.\label{Zfouras}
\end{split}
\end{equation}
Using (\ref{pfourasym}) and (\ref{Zfouras}) as before, we get  
\begin{equation*}
\begin{split}
E_{\beta_4,t}\left[<\zeta,x_t>^2\right]&=Z^{-1}_{\beta_4,t}\sum_{x\in \bf{Z^4} }p_{\beta_4}(t,0,x)<\zeta,x>^2\\ 
&=Z^{-1}_{\beta_4,t}\left(|\zeta|^2t+\beta_4)|\zeta|^2\int_0^t (t-s)p_{\beta_4}(s,0,0)ds\right)\\ 
&\sim\left(c_4\beta_4 \frac{t}{\ln t}\right)^{-1}|\zeta|^2t\left(1 +c_4\beta_4 \int_e^t(1-\frac st)\ln s^{-1}du \right),\,\,\ t\rightarrow \infty\\ 
&\sim  |\zeta|^2t.
\end{split}
\end{equation*}
A similar argument handles the covariance claim.

For the distribution of $\sigma_t/t$ when $d=4$ we observe as before, for $s\in(0,1)$ and $\epsilon>0$ small enough so that $s+\epsilon\in (0,1)$  
\begin{eqnarray*}
\begin{split}
P_{\beta_4,t} \left(\sigma_t /t\in (s,(s+\epsilon))\right)= &8P_{\beta_4,t} \left(x_{st}=0,x_{(s+\epsilon)t}=v_1,x_u\not=0, (s+\epsilon)t\le u\le t\right)+o(\epsilon)\\
 \stackrel{\sim}{\epsilon\to0}&\frac{8E\lt[e^{\beta_4\int_0^t\delta_0(x_u)du}\delta_0(x_{st})\rt]P_{v_1}(x_u\not=0,\,0\le u\le (1-(s+\epsilon))t)(\epsilon t)}{Z_{\beta_4,t}}
 \end{split}
\end{eqnarray*}
where the term $\epsilon t$ arises from the rate of jumping from $0$ to the unit vector $v_1\in {\bf{Z^4}}$ in the time interval $(st,(s+\epsilon)t).$ 
The term
 \begin{eqnarray*}
 P_{v_1}(x_u\not=0,\,0\le u\le (1-(s+\epsilon))t)\to e_4\in(0,1)
\end{eqnarray*}
due to the transience of the $4$ dimensional random walk.  
Thus,
\begin{eqnarray*}
\begin{split}
P_{\beta_4,t} \left(\frac{\sigma_t}{t} \in ds\right)/ds\sim&Z^{-1}_{\beta_4,t}E\lt[e^{\beta_4\int_0^t\delta_0(x_u)du}\delta_0(x_{st})\rt]8e_4t\\
=&Z^{-1}_{\beta_4,t}p_{\beta_4}(st,0,0)8e_4t.
\end{split}
\end{eqnarray*}
By   (\ref{pfourasym}) and (\ref{Zfouras}) we have for $s\in (0,1),$
\begin{eqnarray*}
\lim_{t\to\infty}Z^{-1}_{\beta_4,t}p_{\beta_4}(st,0,0)8e_4t=&\lim_{t\to\infty}\frac{c_4\,8e_4t/\ln st}{c_4\beta_4 t/\ln t}\\
=&\frac{8e_4}{\beta_4}
\end{eqnarray*}
from which we derive that $e_4=8\beta_4$ and so 
\[\lim_{t\to\infty}P_{\beta_4,t} \left(\frac{\sigma_t}{t} \in ds\right)/ds=1,\,\,s\in(0,1)\]
as desired.

Again, the limiting distribution in the critical case $\beta=\beta_4$ is a mixture of Gaussians, but with a different mixture in $d=4$ than in $d=3.$ 
Evaluating $\hat{p}_{\beta_4}(t,0,\cdot)$ at $\frac{\phi}{\sqrt t}$ with $\phi\in {\bf{R^4}}$ and making use of (\ref{pbhatpzero}) together with the central limit theorem for the simple symmetric continuous time random walk, we get
\begin{eqnarray*}
\hat{p}_{\beta_4}(t,0,\frac{\phi}{\sqrt t})=e^{-|\phi| ^2}(1+o(1))+c_4\beta_4 \int_0^t (1+o(1))e^{-|\phi|^2(1-\frac st)}(\ln s\vee1)^{-1}ds
\end{eqnarray*}
After normalization by $Z_{\beta_4,t}\sim c_4 \beta_4 \frac {t}{\ln t}$ it follows that
$$E_{\beta_4,t} \left[e^{i<\phi,\,\frac{x_t }{\sqrt t}>}\right]=\frac{\hat{p}_{\beta_4}(t,0,\frac{\phi}{\sqrt t})}{Z_{\beta_4,t}}\rightarrow \int_0^1e^{-|\phi|^2(1-u)}du.$$
That completes the proof.
\end{proof}

The proof of ${\bf{Theorem (\ref{globular})}}$ requires the following lemma.

\begin{lemma}\label{zzzzz}
For $d\ge 5,$ if $\psi_{\beta_d}$ is normalized so that $\psi_{\beta_d}(0)=1$ then
\begin{eqnarray}\label{Zratio}
\lim_{t\to\infty}\frac{Z_{{\beta_d},t}(x)}{Z_{{\beta_d},t}}=\psi_{\beta_d}(x).
\end{eqnarray}
\end{lemma}
\begin{proof}
Define the hitting time $T=\inf\{t\ge0:x_t=0\}.$ We claim that $\psi_{\beta_d}(x)=P^x(T<\infty).$ Writing $u(x)=P^x(T<\infty)$ we claim $\beta_d=2 de_d$ implies that $H_ {\beta_d}u=0.$ If so, the condition $u(0)=\psi_{\beta_d}(0)=1$ would establish our claim by uniqueness. Since $\Delta u(x)=0,\,x\in {\bf{Z^d}}$ we only need to check that $\Delta u(0)=-2de_du(0)=-2de_d.$
But since $u(y)=1-e_d$ for the unit vectors $y\in{\bf{Z^d}}$ 
\begin{eqnarray*}
\Delta u(0)=&\sum_{y:|y|=1}(u(y)-u(0))\\
=&\sum_{y:|y|=1}(-e_d)\\
=&-2de_d
\end{eqnarray*}
as desired. So to prove the lemma, we need to show the ratio of partition functions converges to $P^x(T<\infty).$

For $d\ge5$ at $\beta=\beta_d$ by   (\ref{pdasym}), (\ref{Zrelp}) and (\ref {phatformula}), we have
\begin{equation}
\begin{split}\label{zdasym}
Z_{\beta_d,t}&=\sum_{x\in \zd}p_{\beta_d} (t,0,x)\\ 
&=\hat{p}_{\beta_d}(t,0,0)\\ 
&= 1+\beta_d \int_0^t p_{\beta_d} (s,0,0)ds\\ 
&\sim c_d\beta_d t, \,t\rightarrow \infty.
\end{split}
\end{equation}

From (\ref{zdasym}) it follows readily that for each  $u,$
\begin{eqnarray}\label{Zratiolimit}
\lim_{t\to\infty} \frac{Z_{\beta_d,t-u}}{Z_{\beta_d,t}}=1.
\end{eqnarray}
To finish the proof, we use the strong Markov property, (\ref{tincr}) ,(\ref{Zratiolimit}) and the dominated convergence theorem to get 
 \begin{eqnarray*}
 \lim_{t\to\infty}\frac{Z_{{\beta_d},t}(x)}{Z_{{\beta_d},t}}=&\lim_{t\to\infty}\int_0^t \frac{Z_{{\beta_d},t-u}}{Z_{{\beta_d},t}}P^x(T\in du)\\
=&P^x(T<\infty).
\end{eqnarray*}

\end{proof}

\begin{proof} (Theorem (\ref{globular}))
The semigroup generated by $p_{\beta_d}$ is
\begin{eqnarray}\label{qsemi}
Q_t f(x)=\sum_{y\in \bf{Z^d}}p_{\beta_d}(t,x,y)f(y) =e^{tH_{\beta_d}}f(x)
\end{eqnarray}
and $Q_t$ acts on the space of bounded functions. Since $\psi_{\beta_d}$ is the eigenfunction corresponding to $\lambda_0({\beta_d})=0,$ (\ref{qsemi}) implies
\begin{eqnarray}\label{qpsi}
Q_t\psi_{\beta_d}(x)=\psi_{\beta_d}(x).
\end{eqnarray}
The computation of the limiting transition probabilities follows from (\ref{Zratio}) by recalling that $\psi_{\beta_d}(0)=1$ and considering for $s+u<t,$
\begin{eqnarray*}
\begin{split}
P_{\beta_d,t}(x_s=x,\,x_{s+u}= y)=&Z_{\beta_d,t}^{-1}E\left[e^{\beta_d\int_0^t\delta_0(x_r)dr},x_s=x,\,x_{s+u}= y\right] \\
=&  Z_{\beta_d,t}^{-1}E\left[e^{\beta_d\int_0^s\delta_0(x_r)dr},\,x_s=x\right]E^x\left[e^{\beta_d\int_0^u\delta_0(x_r)dr},\,x_u=y\right]\\
\times&E^y\left[e^{\beta_d\int_0^{t-s-u}\delta_0(x_)dr}\right]\\
=&\frac{p_{\beta_d}(s,0,x)Z_{\beta_d,t}(x)}{Z_{\beta_d,t}}\frac{p_{\beta_d}(u,x,y)Z_{\beta_d,t-s-u}(y)}{Z_{\beta_d,t}(x)}\\
\to&     \frac{p_{\beta_d}(s,0,x)\psi_{\beta_d}(x)}{\psi_{\beta_d}(0)}\frac{p_{\beta_d}(u,x,y)\psi_{\beta_d}(y)}{\psi_{\beta_d}(x)}
\end{split}
\end{eqnarray*}
which establishes (\ref{transdens}).

Now the kernel 
$$r_{\beta_d}(t,x,y)=\frac{p_{\beta_d}(t,x,y)\psi_{\beta_d}(y)}{\psi_{\beta_d}(x)}$$
generates a semigroup, which we denote by $R_t$ and by (\ref{qpsi}), $R_t\bf{1}=\bf{1}.$
The generator $A_{\beta_d}$ of $R_t$ as calculated by the formula 
$$A_{\beta_d} f(x)=\lim_{h\searrow 0} \frac{R_hf(x)-f(x)}{h}$$
results in the expression 
$$A_{\beta_d} f(x)=\sum_{|y-x|=1} a_d(x,y)(f(y)-f(x)),$$ as claimed in the $\bf{Theorem}.$
Using the asymptotic formula, which comes from the spectral theorem (recall the a.c. part of the spectrum of $H_{\beta_d}$ is $[-4d,0]$)
$$p_{\beta_d}(t,x,y)\sim \frac{\psi_{\beta_d}(x)\psi_{\beta_d}(y)}{||\psi_{\beta_d}||^2_{L^2(\zd)}},\,t\to\infty$$
together with the definition of $r_{\beta_d}(t,x,y)$ we see that both
\[\lim_{t \rightarrow \infty} r_{\beta_d}(t,x,y) = \frac{\psi_{\beta_d}^2(y)}{||\psi_{\beta_d}||^2_{L^2(\zd)}}.\]
and
\[\sum_{x\in {\bf{Z^d}}}\psi_{\beta_d}^2(x)r_{\beta_d}(t,x,y)   =\psi_{\beta_d}^2(y).\]

In order to obtain weak convergence of the measures $P_{\beta_d,t},\,t\to\infty,$ on the space of trajectories $\Sigma_s$ we need to establish tightness. According to \cite{Bill}, this requires control of the oscillations. Set
$$\omega_s(x,[t_{i-1},t_i))\equiv\sup_{u,v\in [t_{i-1},t_i)}|x_u-x_v|$$
and
$$\omega'_s(x,\delta)\equiv\inf_{\{t_i\}}\max_{1\le i\le r}\omega_s(x,[t_{i-1},t_i)),$$
where the $\inf$ is taken over finite sets $\{t_i\}$ such that 
\begin{equation*}\begin{split}
0&=t_0<t_1<\dots<t_r=s\\
t_i&-t_{i-1}>\delta.
\end{split}
\end{equation*}
Then, since our paths all start at $0$ under $P_{\beta_d,t}$  tightness will follow provided for each positive $\epsilon$ and $\eta$ there exists a $\delta\in (0,1)$and $T_0$ such that 
\begin{equation}\label{tight}
P_{\beta_d,t}\left(\omega'_s(x,\delta)\ge \epsilon\right)\le \eta,\,\,\,t\ge T_0.
\end{equation}
But, using (\ref{zerobest}) and (\ref{tincr}), we have
\begin{eqnarray}\label{estimate}
\begin{split}
P_{\beta_d,t}\left(  \omega'_s(x,\delta)\ge\epsilon \right)= &Z^{-1}_{\beta_d,t}E^0\left[ e^{\beta_d \int_0^t\delta_0(x_u)ds} 1_{\{\omega'_s(x,\delta)\ge \epsilon\}}\right]\\
=&Z^{-1}_{\beta_d,t}E^0\left[e^{\beta_d\int_0^s\delta_0(x_u)du}1_{\{\omega'_s(x,\delta)\ge \epsilon\}}E^{x_s}\left[e^{\beta_d \int_0^{t-s}\delta_0(x_u)du}\right]\right]\\
\le& E^0\left[e^{\beta_d s}1_{\{\omega'_s(x,\delta)\ge \epsilon\}}\right]\frac{Z_{\beta_d,t-s}}{Z_{\beta_d,t}}\\
\le& e^{\beta_d s} P^0\left(\omega'_s(x,\delta)\ge \epsilon\right).
\end{split}
\end{eqnarray}

Thus, we see from (\ref{estimate}) that (\ref{tight}) follows from the fact that (\ref{tight}) holds for $P^0.$
This proves the convergence in distribution of the process on $\Sigma_s$ under $P_{\beta_d,t}$ as $t\rightarrow \infty.$
\end{proof}

\bigskip
\noindent {\bf M. Cranston.} Department of Mathematics, 
	University of California-Irvine, Irvine, CA 92697-3875,
	{\mbox{mcransto@gmail.com}}\\
	
\noindent {\bf S. Molchanov.} Department of Mathematics, University of NC,Charlotte,
	Charlotte, NC 28223,
	{\mbox{davar@math.utah.edu}}\\

\end{document}